\documentclass[12pt]{amsart}

\usepackage[mathscr]{eucal}
\usepackage{amsmath,amssymb,amscd,amsthm}
\usepackage{color}
\usepackage{epsfig}
\newtheorem{theorem}{Theorem}

\newtheorem{definition}{Definition}

\newtheorem{criterion}{Criterion}

\theoremstyle{definition}

\theoremstyle{plane}

\def \beq{ \begin{equation} }
\def \eeq{\end{equation}}

\title{Polygonal homographic orbits of the curved $n$-body problem}
\begin{document}
\maketitle
\markboth{Florin Diacu}{Polygonal homographic orbits of the curved $n$-body problem}
\author{\begin{center}
Florin Diacu\\
\smallskip
{\footnotesize Pacific Institute for the Mathematical Sciences\\
and\\
Department of Mathematics and Statistics\\
University of Victoria\\
P.O.~Box 3060 STN CSC\\
Victoria, BC, Canada, V8W 3R4\\
diacu@math.uvic.ca\\
}\end{center}

}

\vskip0.5cm

\begin{center}
\today
\end{center}

\begin{abstract}
In the $2$-dimensional $n$-body problem, $n\ge 3$, in spaces of constant curvature, $\kappa\ne 0$, we study polygonal homographic solutions. We first provide necessary and sufficient conditions for the existence of these orbits and then consider the case of regular polygons. We further use this criterion to show that, for any $n\ge 3$, the regular $n$-gon is a polygonal homographic orbit if and only if all masses are equal. Then we prove the existence of relative equilibria of non-equal masses on the sphere of curvature $\kappa>0$ for $n=3$ in the case of scalene triangles. Such triangular relative equilibria occur only along fixed geodesics and are generated from fixed points of the sphere. Finally, through a classification of the isosceles case, we prove that not any three masses can form a triangular relative equilibrium. 
\end{abstract}
 
\section{Introduction}
 
We study here the curved $n$-body problem, defined as the motion of $n$ point particles of masses $m_1,m_2,\dots,m_n>0$ in spaces of constant curvature, $\kappa\ne 0$, under the influence of a natural extension of Newton's gravitational law. The potential defining this law is given by the force function $U_\kappa$, whose expression appears in \eqref{forcef}. The corresponding potential, given by $-U_\kappa$, preserves the basic properties of its Euclidean analogue: it is a harmonic function in the $3$-dimensional space, i.e.\ it satisfies Laplace's equation, and generates a central field in which bounded orbits are closed, \cite{Lie2}, in agreement with an old result proved by Joseph Louis Bertrand for the Euclidean case, \cite{Ber}.

The research direction we are following here started in the 1830s, when J\'anos Bolyai and Nikolai Lobachevsky independently proposed a  curved 2-body problem in the hyperbolic space ${\bf H}^3$, given by a force proportional to the inverse of the area of the sphere of radius equal to the distance between bodies. This problem was studied by top mathematicians, such as Lejeune Dirichlet, Ernest Schering, \cite{Sche}, \cite{Sche1}, Wilhem Killing, \cite{Kil1}, \cite{Kil10}, \cite{Kil2}, and Heinrich Liebmann, \cite{Lie1}, \cite{Lie2}, \cite{Lie3}. Schering was the one who came up with an analytic form of the potential, which can be expressed in terms of the cotangent of the distance,
for $\kappa>0$, and the hyperbolic cotangent of the distance, for $\kappa<0$.
The Newtonian law is recovered when $\kappa\to 0$. Recently, Jos\'e Cari\~nena, Manuel Ra\~nada, and Mariano Santander, \cite{Car}, proved several new results
for this  $2$-body problem  and revisited some old properties with the help of modern methods. The study of the quantum analogue of the curved $2$-body problem was proposed by Erwin Schr\"odinger, \cite{Schr}, and continued by Leopold Infeld, \cite{Inf}, and Alfred Schild, \cite{InfS}. 

Other attempts at extending the Newtonian case to spaces of constant curvature, such as the efforts of Rudolph Lipschitz, \cite{Lip}, did not survive, mostly because the proposed potentials lacked the basic physical properties mentioned above. Unlike Liebmann, who showed that all the fundamental orbits of the Kepler problem have analogues in curved space, Lipschitz could not obtain explicit solutions; he only succeeded to express the orbits in terms of elliptic integrals. 

In the direction of research proposed by Bolyai and Lobachevsky, the generalization
of the equations of motion to any $n\ge 2$ was given in \cite{Diacu1}, a paper posted in
arXiv, but submitted for publication as two connected papers, \cite{Diacu001} and \cite{Diacu002}, in which we obtained a unified framework of approaching the problem for any $\kappa\ne 0$. We also proved there the existence of several classes of relative equilibria, including the Lagrangian orbits, i.e.\ the case $n=3$ of the equilateral triangle. Relative equilibria are orbits for which the configuration of the system remains congruent with itself for all time, so the mutual distances between bodies are constant during the motion. 

It is worth mentioning that the study of the curved $n$-body problem, for $n\ge 3$,  might help us better understand the nature of the physical space. Gauss allegedly tried to determine the geometry of the universe by measuring the angles of a triangle formed by the peaks of three mountains. Even if the goal of his topographic measurements was different from what anecdotical history attributes to him (see \cite{Mill}), this method of deciding the nature of space remains valid for astronomical distances. But since we cannot measure the angles of cosmic triangles, we could alternatively check whether certain (potentially observable) celestial motions occur in the universe, and
thus decide whether the physical space has negative, zero, or positive curvature.

Specifically, we showed in \cite{Diacu1} and \cite{Diacu001} that while Lagrangian orbits of non-equal masses are known to occur for $\kappa=0$, they must have equal masses for $\kappa\ne 0$. Since Lagrangian solutions of non-equal masses exist in our solar system (such as the triangle formed by the Sun, Jupiter, and the Trojan asteroids), we can conclude that, if assumed to have constant curvature, the physical space is Euclidean for distances of the order $10^1$ AU. The discovery of new orbits of the curved $n$-body problem might help us better understand the large-scale geometry of the universe.

The most recent papers on the curved $n$-body problem deal either with singularities, \cite{Diacu1bis}, \cite{Diacu002}, a subject we will not approach here, or with homographic solutions and, in particular, with homothetic orbits and relative equilibria, \cite{Diacu1}, \cite{Diacu002}, \cite{Diacu4}. Homographic solutions are orbits whose configuration remains similar to itself all along the motion. In particular, when rotation takes place without expansion or contraction, the homographic orbits are called relative equilibria, as mentioned earlier. They behave like rigid bodies, maintaining constant mutual distances. Homothetic solutions are homographic orbits that experience expansion and/or contraction, but no rotation.

The homograpic solutions can be put in a broader perspective. They are also the object of Saari's conjecture, \cite{Saari}, \cite{Diacu2}, which we partially solved for the curved $n$-body problem, \cite{Diacu1}, \cite{Diacu001}, as well as Saari's homographic conjecture, \cite{Saari}, \cite{Diacu3}. Both have recently generated a lot of interest in classical celestial mechanics (see the references in \cite{Diacu2}, \cite{Diacu3}) and are still unsolved in general. 

In the classical Newtonian case, \cite{Win}, as well as in more general classical contexts, \cite{Diacu0}, the standard concept for understanding homographic solutions is that of central configuration. We will not employ it here since most computations appear to be simpler without using it. The reason for these complications is connected to the absence of the integrals of the centre of mass and linear momentum from
the curved $n$-body problem. These integrals seem to be specific only to Euclidean space. Indeed, $n$-body problems derived by discretizing Einstein's field equations, as obtained by Tullio Levi-Civita, \cite{Civ}, \cite{Civita}, Albert Einstein, Leopold Infeld, and Banesh Hoffmann, \cite{Ein},  and Vladimir Fock, \cite{Fock}, also lack such integrals.

In this paper we study polygonal homographic orbits of the 2-dimension\-al curved $n$-body problem. In Section 2, we introduce the notation and the equations of motion as
well as their first integrals. In Section 3 we define polygonal homographic orbits and
their basic particular cases: homothetic orbits and relative equilibria. We
also provide a motivation for our definitions, which are given in the spirit of the 
Euclidean case. In Section 4, we state and prove necessary and sufficient conditions for the existence of polygonal homographic orbits. Section 5 is dedicated to the study
of regular polygons. We show, on one hand, that if the masses are equal, the
regular $n$-gon is a solution of the equations of motion for any $n\ge 3$. On the
other hand, we prove that regular $n$-gons can be solutions only when the masses
are equal. These results extend the Perko-Walter-Elmabsout theorem, \cite{Perko}, \cite{Elma}, to spaces of non-zero constant curvature. In Section 6, we focus on the case $n=3$ and prove that the equations
of motion admit no homographic orbits if the triangle is not equilateral. For $\kappa>0$
this result is true as long as the homographic orbit is not a relative equilibrium rotating
along a great circle of the sphere. In Section 7 we discuss the orbits omitted in the
previous section, namely the relative equilibria that move along a geodesic of the
sphere. We prove the existence of relative equilibria of non-equal masses in the case of scalene triangles. Such triangular relative equilibria occur only along fixed geodesics and are generated from fixed points on the sphere. Finally we show that not any three masses can form a triangular relative equilibrium by providing a large class of counterexamples in the isosceles case.

\section{Equations of motion}\label{equations}

In this section we introduce the equations of motion of the curved $n$-body problem 
on $2$-dimensional manifolds of constant curvature, namely spheres embedded in $\mathbb{R}^3$, for $\kappa>0$, and the upper sheets of hyperboloids of two sheets\footnote{The upper sheet of the hyperboloid of two sheets corresponds to Weierstrass's model of hyperbolic geometry (see Appendix in \cite{Diacu1} or \cite{Diacu001}).} embedded in the Minkovski space ${\mathbb{R}}^{2,1}$, for $\kappa<0$. 

Consider the point particles (bodies) of masses $m_1, m_2,\dots, m_n>0$ in $\mathbb{R}^3$, for $\kappa>0$,
and in ${\mathbb{R}}^{2,1}$, for $\kappa<0$, whose positions are given by the vectors
${\bf q}_i=(x_i,y_i,z_i), \ i=\overline{1,n}$. Let ${\bf q}=
({\bf q}_1, {\bf q}_2,\dots,{\bf q}_n)$ be the configuration of the system,
and ${\bf p}=({\bf p}_1, {\bf p}_2,\dots,{\bf p}_n)$, with ${\bf p}_i=m_i\dot{\bf q}_i,\ i=\overline{1,n}$,
representing the momentum.
We define the gradient operator with respect to the vector ${\bf q}_i$ as
$$\widetilde\nabla_{{\bf q}_i}=(\partial_{x_i},\partial_{y_i},\sigma\partial_{z_i}),$$
where $\sigma$ is the signum function,
\begin{equation}
\sigma=
\begin{cases}
+1, \ \ {\rm for} \ \ \kappa>0\cr
-1, \ \ {\rm for} \ \ \kappa<0,\cr
\end{cases}\label{sigma}
\end{equation}
and let $\widetilde\nabla=
(\widetilde\nabla_{{\bf q}_1},\widetilde\nabla_{{\bf q}_2},\dots,\widetilde\nabla_{{\bf q}_n})$.
For the 3-dimensional vectors ${\bf a}=(a_x,a_y,a_z)$ and ${\bf b}=(b_x,b_y,b_z)$, 
we define the inner product
\begin{equation}
{\bf a}\odot{\bf b}:=(a_xb_x+a_yb_y+\sigma a_zb_z)
\label{dotpr}
\end{equation}
and the cross product
\begin{equation}
{\bf a}\otimes{\bf b}:=(a_yb_z-a_zb_y, a_zb_x-a_xb_z, 
\sigma(a_xb_y-a_yb_x)).
\end{equation}

The Hamiltonian function of the system describing the motion of the $n$-body problem in spaces of constant curvature is
$$H_\kappa({\bf q},{\bf p})=T_\kappa({\bf q},{\bf p})-U_\kappa({\bf q}),$$
where 
$$
T_\kappa({\bf q},{\bf p})={1\over 2}\sum_{i=1}^nm_i^{-1}({\bf p}_i\odot{\bf p}_i)(\kappa{\bf q}_i\odot{\bf q}_i)
$$
defines the kinetic energy and
\begin{equation}
U_\kappa({\bf q})=\sum_{1\le i<j\le n}{m_im_j
|\kappa|^{1/2}{\kappa{\bf q}_i\odot{\bf q}_j}\over
[\sigma(\kappa{\bf q}_i
\odot{\bf q}_i)(\kappa{\bf q}_j\odot{\bf q}_j)-\sigma({\kappa{\bf q}_i\odot{\bf q}_j
})^2]^{1/2}}
\label{forcef}
\end{equation}
is the force function, $-U_\kappa$ representing the potential energy\footnote{In \cite{Diacu1} and \cite{Diacu001}, we showed how this expression of $U_\kappa$ follows from the cotangent potential for $\kappa\ne 0$, and that $U_0$ is the Newtonian potential of the Euclidean problem, obtained as $\kappa\to 0$.}. Then the Hamiltonian form of the equations of motion is given by the system
\begin{equation}
\begin{cases}
\dot{\bf q}_i=
m_i^{-1}{\bf p}_i,\cr
\dot{\bf p}_i=\widetilde\nabla_{{\bf q}_i}U_\kappa({\bf q})-m_i^{-1}\kappa({\bf p}_i\odot{\bf p}_i)
{\bf q}_i, \ \  i=\overline{1,n}, \ \kappa\ne 0,
\label{Ham}
\end{cases}
\end{equation}
where the gradient of the force function has the expression
\begin{equation}
{\widetilde\nabla}_{{\bf q}_i}U_\kappa({\bf q})=\sum_{\substack{j=1\\ j\ne i}}^n{m_im_j|\kappa|^{3/2}(\kappa{\bf q}_j\odot{\bf q}_j)[(\kappa{\bf q}_i\odot{\bf q}_i){\bf q}_j-(\kappa{\bf q}_i\odot{\bf q}_j){\bf q}_i]\over
[\sigma(\kappa{\bf q}_i
\odot{\bf q}_i)(\kappa{\bf q}_j\odot{\bf q}_j)-\sigma({\kappa{\bf q}_i\odot{\bf q}_j
})^2]^{3/2}}.
\label{gradient}
\end{equation}
The motion is confined to the surface of nonzero constant curvature $\kappa$, i.e.\ $({\bf q},{\bf p})\in {\bf T}^*({\bf M}_\kappa^2)^n$,  where ${\bf T}^*({\bf M}_\kappa^2)^n$ is the cotangent bundle of the configuration space $({\bf M}^2_\kappa)^n$, and
$$
{\bf M}^2_\kappa=\{(x,y,z)\in\mathbb{R}^3\ |\ \kappa(x^2+y^2+\sigma z^2)=1\}.
$$ 
In particular, ${\bf M}^2_1={\bf S}^2$ is the 2-dimensional sphere, and ${\bf M}^2_{-1}={\bf H}^2$ is the 2-dimensional hyperbolic plane, represented by the upper sheet of the hyperboloid of two sheets (see the Appendix of \cite{Diacu1} or \cite{Diacu001} for more details). 
We will also denote ${\bf M}^2_\kappa$ by ${\bf S}^2_\kappa$ for $\kappa>0$
and by ${\bf H}^2_\kappa$ for $\kappa<0$.

Notice that the $n$ constraints given by $\kappa{\bf q}_i\odot{\bf q}_i=1, i=\overline{1,n},$ imply that ${\bf q}_i\odot{\bf p}_i=0$, so the $6n$-dimensional system \eqref{Ham} has  $2n$ constraints.
The Hamiltonian function provides the integral of energy,
$$
H_\kappa({\bf q},{\bf p})=h,
$$
where $h$ is the energy constant. Equations \eqref{Ham} also have the three integrals 
of the angular momentum,
\begin{equation}
\sum_{i=1}^n{\bf q}_i\otimes{\bf p}_i={\bf c},
\label{ang}
\end{equation}
where ${\bf c}=(\alpha, \beta, \gamma)$ is a constant vector. Unlike in the
Euclidean case, there are no integrals of the center of mass and linear
momentum. Their absence complicates the study of the problem
since many of the standard methods don't apply anymore.

Using the fact that $\kappa{\bf q}_i\odot{\bf q}_i=1$ for $i=\overline{1,n}$, we can write system
\eqref{Ham} as
\begin{equation}
\ddot{\bf q}_i=\sum_{\substack{j=1\\ j\ne i}}^n{m_j|\kappa|^{3/2}[{\bf q}_j-(\kappa{\bf q}_i\odot{\bf q}_j){\bf q}_i]\over
[\sigma-\sigma({\kappa{\bf q}_i\odot{\bf q}_j
})^2]^{3/2}}-(\kappa\dot{\bf q}_i\odot\dot{\bf q}_i){\bf q}_i, \ \ i=\overline{1,n},
\label{second}
\end{equation}
which is the form of the equations of motion we will use in this paper. The
sums on the right hand side of the above equations represent the gradient of the potential. When $\kappa\to 0$, both the
sphere (for $\kappa\to 0, \kappa>0$) and the upper sheet of the hyperboloid
of two sheets (for $\kappa\to 0, \kappa<0$) become planes at infinity,
relative to the centre of the frame. The segments through the origin of
the frame whose angle measures the distance between two points on the curved surface become parallel and infinite, so the distance in the limit plane is the 
Euclidean distance. Consequently the potential tends to the Newtonian potential as $\kappa\to 0$ (see \cite{Diacu1} or \cite{Diacu001} for more details).  The terms involving the velocities occur because of the constraints imposed by the curvature. They vanish when $\kappa\to 0$.


\section{Polygonal homographic orbits}

In this section we define the polygonal homographic solutions of the curved $n$-body
problem as well as two remarkable subclasses of solutions: the polygonal homothetic orbits and the polygonal relative equilibria. Then we justify the content of these definitions.

\begin{definition}
A solution of equations \eqref{second}, which describe the curved $n$-body problem, is called polygonal homographic if the bodies of masses $m_1,m_2,\dots, m_n>0$, with 
$n\ge 3$, form a polygon that is orthogonal to the $z$ axis and remains similar to itself for all 
time $t$. 
\label{deflag}
\end{definition}

According to Definition \ref{deflag}, the size of a polygonal homographic solution can vary, but its shape remains the same. Notice that we imposed the condition that
the plane of the polygon is always perpendicular to the $z$ axis. This condition is equivalent to saying that all masses have the same coordinate $z(t)$, which may vary in time. This condition is not imposed for mere simplicity but because polygonal homographic solutions may not exist without it. Though a complete proof of this conjecture is still eluding us, we will explain later in this section why we think this property is true.

We can represent a polygonal homographic solution of the curved $n$-body problem 
in the form
 \begin{equation}
{\bf q}
=({\bf q}_1,\dots, {\bf q}_n),\ \  {\bf q}_i=(x_i,y_i,z_i),
\label{polysol}
\end{equation}
$$
x_i=r\cos(\omega+\alpha_i),\  y_i=r\sin(\omega+\alpha_i),\ z_i=z,\ i=\overline{1,n},
$$
where $0\le\alpha_1<\alpha_2<\dots\alpha_n<2\pi$ are constants; the function $z=z(t)$ satisfies $z^2=\sigma\kappa^{-1}-\sigma r^2$; $\sigma$ is the signum
function defined in \eqref{sigma}; $r:=r(t)$ is the {\it size function}; and $\omega:=\omega(t)$ is the {\it angular function}.

Indeed, for every time $t$, we have that $x_i^2(t)+y_i^2(t)+\sigma z_i^2(t)=\kappa^{-1},\ i=\overline{1,n}$, which means that the bodies move on the surface ${\bf M}_{\kappa}^2$, and the angles between any two bodies, seen from the centre of the circle containing the polygon, are invariant in time. Therefore representation \eqref{polysol} of the polygonal homographic orbits agrees with Definition \ref{deflag}.

\begin{definition}
A polygonal homographic solution of equations \eqref{second}, which describe the curved $n$-body problem, is called polygonal homothetic if the polygon having at its vertices the bodies of masses $m_1,m_2,\dots, m_n>0$, with $n\ge 3,$ expands or contracts, but does not rotate around the $z$ axis. 
\end{definition}

In terms of representation \eqref{polysol}, a polygonal homographic solution is polygonal homothetic if $\omega(t)$ is constant, but $r(t)$ is not. Such orbits occur, for instance, when $n$ bodies of equal masses, lying initially at the vertices of a regular polygon inscribed in a non-geodesic circle of the sphere, are released with zero initial velocities, to end up in a total collision.

\begin{definition}
A polygonal homographic solution of equations \eqref{second}, which describe the curved $n$-body problem, is called a polygonal relative equilibrium if the polygon having at its vertices the bodies of masses $m_1,m_2,\dots,$ $m_n>0$, with $n\ge 3$, rotates around the $z$ axis and maintains fixed mutual distances. 
\end{definition}

In terms of representation \eqref{polysol}, a polygonal relative equilibrium occurs when $r(t)$ is constant, but $\omega(t)$ is not. These orbits have a rich recent history and have been extensively studied in the general context of geometric mechanics (see, e.g., \cite{abra}, \cite{tudor}).

We will further loosely use the terms ``dynamical polygon'' or ``dynamical $n$-gon'' to describe any of the polygonal homographic, homothetic, or relative equilibrium orbits that will occur in this paper. These terms will also occur when we check potential
solutions, prior to knowing whether they satisfy the equations of motion.

Using the concept of relative equilibrium, we can now provide a justification for introducing the orthogonality condition in Definitions 1, 2, and 3. Consider a polygon formed by the bodies of masses $m_1,m_2,\dots, m_n>0$, which move on the surface ${\bf M}_k^2$ according to equations \eqref{second}. Unlike in representation \eqref{polysol}, the plane of the polygon can have any angle (not only $\pi/2$) relative to the $z$ axis. We assume this angle to be constant in time. We assign initial velocities such that the dynamical polygon rotates around the $z$ axis, which passes at all times through the same point inside the polygon. The Principal Axis Theorem (see Appendix in \cite{Diacu1} or \cite{Diacu001}) guarantees the validity of this scenario without any loss of generality.

Let us now seek necessary conditions for the existence of the above described motion. Notice that the projection of the polygon to the $xy$ plane is, at any time $t$, a polygon congruent with the projection obtained at $t=0$. We can then describe the motion of the real polygon in terms of the angles of the projected polygon. So let us assume that the relative equilibrium is represented by the coordinates ${\bf q}_i(t)=(x_i(t),y_i(t), z_i(t)),\ i=\overline{1,n}$, given by
\begin{equation}
x_i(t)=r_i\cos(\Omega t+\alpha_i),\ y_i(t)=r_i\sin(\Omega t+\alpha_i),\ z_i(t)=z_i \ {\rm (constant)}
\label{newrep}
\end{equation}
where $0\le\alpha_1<\alpha_2<\dots\alpha_n<2\pi$, $\Omega\ne 0$, and $r_i>0$ are constants; $z_i^2=\sigma\kappa^{-1}-\sigma r_i^2$; and $\sigma$ is the signum function defined in \eqref{sigma}.

For such a solution to exist, it is necessary that the total angular momentum is
the zero vector or a vector parallel with the $z$ axis. Otherwise the angular-momentum vector would rotate around the $z$ axis, in violation of the angular-momentum integrals \eqref{ang}. This means that at least the first two components of the vector $\sum_{i=1}^nm_i{\bf q}_i\otimes\dot{\bf q}_i$ are zero, i.e.
$$
\sum_{i=1}^nm_i(y_i\dot{z}_i-\dot{y}_iz_i)= 
\sum_{i=1}^nm_i(x_i\dot{z}_i-\dot{x}_iz_i)=0.
$$
Using \eqref{newrep} and the fact that $\Omega\ne 0$, these two equations take the form
$$
\sum_{i=1}^n m_ir_iz_i\cos(\Omega t+\alpha_i)=\sum_{i=1}^nm_ir_iz_i\sin(\Omega t+\alpha_i)=0.
$$
In general, the above conditions for the existence of motions described above are not satisfied for all values of $t$. But there are exceptions, such as when all $z_i$ (and implicitly all $r_i$) are equal, all $m_i$ are equal, and the polygon is regular. So the case when all $z_i$ are equal, which implies orthogonality relative to the $z$ axis, seems like a good point to start from. This position is supported by the proof that  Lagrangian solutions must be orthogonal to the $z$ axis (see \cite{Diacu1} or \cite{Diacu001}). These remarks justify our choice of the orthogonality condition in the definitions of this section.


\section{Necessary and sufficient conditions}

The goal of this section is to state and prove two equivalent criteria that provide 
necessary and sufficient conditions for the existence of polygonal homographic solutions of the curved $n$-body problem. The first criterion can be expressed as follows.

\begin{criterion}
Consider $n\ge 3$ bodies of masses $m_1,m_2,\dots, m_n>0$ moving
on the surface ${\bf M}^2_\kappa$. The necessary and sufficient
conditions that a solution of the form \eqref{polysol} is a polygonal 
homographic orbit of equations \eqref{second} are given by the equations 
\begin{equation}
\delta_1=\delta_2=\dots=\delta_n\ \ {\rm and}\ \  \gamma_1=\gamma_2=\dots=\gamma_n,
\label{deltagamma}
\end{equation}
where
\begin{equation}
\delta_i=\sum_{j=1, j\neq i}^nm_j\mu_{ji}, \ \
\gamma_i=\sum_{j=1,j\neq i}^nm_j\nu_{ji},\ \ i=\overline{1,n},
\end{equation}
\begin{equation}\mu_{ji}=\frac{1}{c_{ji}^{1/2}(2-c_{ji}\kappa r^2)^{3/2}}, \ \  \nu_{ji}=\frac{s_{ji}}{c_{ji}^{3/2}(2-c_{ji}\kappa r^2)^{3/2}},
\label{munu}
\end{equation}
\begin{equation}
s_{ji}=\sin(\alpha_j-\alpha_i),\ \ c_{ji}=1-\cos(\alpha_j-\alpha_i), \ \ i,j=\overline{1,n},
\ \ i\ne j.
\label{sin-cos}
\end{equation}
\label{EandU}
\end{criterion}
\begin{proof}
Let us check in equations \eqref{second} a solution of the form \eqref{polysol}. 
For this purpose we first compute that 
$$
\dot{x}_i=\dot{r}\cos(\omega+\alpha_i)-r\dot\omega\sin(\omega+\alpha_i),\ \
\dot{y}_i=\dot{r}\sin(\omega+\alpha_i)+r\dot\omega\cos(\omega+\alpha_i)
$$
$$
\dot{z}_i=-\sigma r\dot{r}(\sigma\kappa^{-1}-\sigma r^2)^{-1/2}.
$$
$$
\ddot{x}_i=(\ddot{r}-r\dot{\omega}^2)\cos(\omega+\alpha_i)-(r\ddot\omega+2\dot{r}
\dot\omega)\sin(\omega+\alpha_i),
$$
$$
\ddot{y}_i=(\ddot{r}-r\dot{\omega}^2)\sin(\omega+\alpha_i)+(r\ddot\omega+2\dot{r}
\dot\omega)\cos(\omega+\alpha_i),
$$
$$
\ddot{z}_i=-\sigma r \ddot{r}(\sigma\kappa^{-1}-\sigma r^2)^{-1/2}-\kappa^{-1}
\dot{r}^2(\sigma\kappa^{-1}-\sigma r^2)^{-3/2}.
$$
Some long but straightforward computations for the $\ddot{z}_i$ component lead 
us to to the equations
\begin{equation}
{\ddot r}=r(1-\kappa r^2)\dot\omega-\frac{\kappa r\dot r^2}{1-\kappa r^2}-\Delta_i,\
\ i=\overline{1,n},
\label{r}
\end{equation}
where 
\begin{equation}
\Delta_i=\sum_{j=1, j\neq i}^n\frac{m_j(1-\kappa r^2)}{c_{ji}^{1/2}r^2(2-c_{ji}\kappa r^2)^{3/2}},\ \ i=\overline{1,n},
\label{Delta}
\end{equation}
and the constants $c_{ji}$ are defined in \eqref{sin-cos}. Notice that, as long as
the bodies stay away from collisions, we have $c_{ji}>0$.

From the equations corresponding to $\ddot{x}_i$ and $\ddot{y}_i$, we further apply the formula for the cosine of the sum of two angles to
$\cos(\omega+\alpha_j)$, viewed as $\cos[(\omega+\alpha_i)+(\alpha_j-\alpha_i)]$.
Then, using the fact that the equations must be satisfied for all times $t$, and comparing the similar terms, we recover equations \eqref{r} as well as obtain some new equations,
\begin{equation}
r \ddot\omega + 2\dot r\dot\omega-\Gamma_i=0, \ \ i=\overline{1,n},
\label{omega}
\end{equation}
where
\begin{equation}
\Gamma_i=\sum_{j=1,j\neq i}^n\frac{m_js_{ji}}{c_{ji}^{3/2}r^2(2-c_{ji}\kappa r^2)^{3/2}},\ \ i=\overline{1,n},
\label{Gamma}
\end{equation}
and the constants $s_{ji}, c_{ji},\ i, j=\overline{1,n}, i\ne j,$ are defined in \eqref{sin-cos}.
Notice that $s_{ji}$ can have any sign, with $s_{ji}=0$ only
if $\alpha_j-\alpha_i=\pm\pi/2$ or $\pm3\pi/2$.

Equations \eqref{r} and \eqref{omega} describe the motion of the polygonal
homographic orbit with the help of the size function, $r$, and the angular
function, $\omega$. But this system of differential equations makes sense only if
the conditions
$$
\Delta_1=\Delta_2=\dots=\Delta_n\ \ {\rm and}\ \  \Gamma_1=\Gamma_2=\dots=\Gamma_n
$$
are satisfied, where the functions $\Delta_i, \Gamma_i, i=\overline{1,n},$ are defined in \eqref{Delta} and \eqref{Gamma}, respectively. Then, for any initial conditions $(r(0),\omega(0))$ that are not singular (i.e.\ are not collisions for any $\kappa\ne 0$, and are not antipodal for $\kappa>0$, see \cite{Diacu1} or \cite{Diacu002}), we are guaranteed the local existence and uniqueness of an analytic solution for the system given by equations \eqref{r} and \eqref{omega}. Since the phase space is a connected manifold, this solution can be analytically extended to a global solution, defined either for all time or until the orbit reaches a singular configuration, as it may happen, for instance, in the case of a homothetic orbit that ends in a total collision.

Multiplying equation \eqref{r} by $r^2/(1-\kappa r^2)$ and equation 
\eqref{omega} by $r^2$, the conditions for the existence of polygonal homographic
orbits reduce to
$$
\delta_1=\delta_2=\dots=\delta_n\ \ {\rm and}\ \  \gamma_1=\gamma_2=\dots=\gamma_n,
$$
as defined in the above statement. This remark completes the proof.
\end{proof}

\noindent{\it Remark}. Notice that, if $\delta$ denotes any $\delta_i$ and
$\gamma$ denotes any $\gamma_i, i=\overline{1,n}$, the functions $\mu_{ji}$ and $\delta$ are always positive, while $\nu_{ji}$ and $\gamma$  can be negative, positive, or zero.

\medskip

We can restate Criterion \ref{EandU} in terms of linear algebraic systems as follows.
The equivalence between Criterion  \ref{EandU} and Criterion \ref{secondCriterion} is obvious.

\begin{criterion}
Consider $n\ge 3$ bodies of masses $m_1,m_2,\dots, m_n>0$ moving
on the surface ${\bf M}^2_\kappa$. The necessary and sufficient
conditions that a solution of the form \eqref{polysol} is a polygonal 
homographic orbit of equations \eqref{second} are that there exist
$\delta>0$ and $\gamma\in{\mathbb R}$ such that the linear systems
\begin{equation}
\Delta{\bf m}^T={\bf d}^T\ \ {\rm and}\ \ \Gamma{\bf m}^T={\bf e}^T
\end{equation}
have the same set of solutions, where
$${\bf m}=(m_1,m_2,\dots,m_n),\ \ {\bf d}=(\underbrace{\delta,\delta,\dots,\delta}_\text{$n$ times}), \ \ {\bf e}=(
\underbrace{\gamma,\gamma,\dots,\gamma}_\text{$n$ times}),$$
$$\Delta=\left[ \begin{array}{ccccc}
0 & \mu_{21} & \mu_{31} & \dots & \mu_{n1} \\
\mu_{12} & 0 & \mu_{32} & \dots & \mu_{n2}\\
\mu_{13} & \mu_{23} & 0 & \dots & \mu_{n3}\\
\vdots & \vdots & \vdots & {} & \vdots \\
\mu_{1n} & \mu_{2n} & \mu_{3n} & \dots & 0\\
\end{array} \right],\ \
\Gamma=\left[ \begin{array}{ccccc}
0 & \nu_{21} & \nu_{31} & \dots & \nu_{n1} \\
\nu_{12} & 0 & \nu_{32} & \dots & \nu_{n2}\\
\nu_{13} & \nu_{23} & 0 & \dots & \nu_{n3}\\
\vdots & \vdots & \vdots & {} & \vdots \\
\nu_{1n} & \nu_{2n} & \nu_{3n} & \dots & 0\\
\end{array} \right],$$
$\mu_{ji}, \nu_{ji},\ i,j=\overline{1,n},\ i\ne j$ are defined in \eqref{munu},
and the upper index ${}^T$denotes the transpose of a vector or matrix.
\label{secondCriterion}
\end{criterion}

\noindent{\it Remark}. Notice that $\Delta$ is symmetric, i.e.\ $\Delta^T=\Delta$, whereas $\Gamma$ is skew-symmetric, i.e.\ $\Gamma^T=-\Gamma$. In the
next section we will see that both $\Delta$ and $\Gamma$ are circulant
matrices, i.e.\ each row vector is rotated one element to the right relative
to the previous (upper) row vector.


\section{Regular polygons}

In this section we will study the case of regular polygons and prove that the curved
$n$-body problem, with $n\ge 3$, admits regular polygonal homographic orbits if and only if all the masses are equal. This result is valid only for $n\ge 4$ in the Euclidean case,
where it is known as the Perko-Walter-Elmabsout theorem. The classical Lagrangian orbits, given by the dynamic equilateral triangle, allow non-equal masses, a situation that seems to be specific to the Euclidean space, most likely because some symmetries are lost if the curvature is not zero (see also \cite{Diacu0}). To emphasize certain details of our result, we will separately consider the direct and converse components of this theorem.

\begin{theorem}
Consider the curved $n$-body problem, $n\ge 3$, given by system \eqref{second}.
If $n$ bodies of equal masses, $m:=m_1=m_2=\dots=m_n$, lie initially at the vertices of a regular $n$-gon parallel with the $xy$ plane, then there is a class of initial velocities for
which the corresponding solutions are homographic. These orbits also satisfy the equalities $\gamma_1=\gamma_2=\dots=\gamma_n=0$.  
\label{eqmasses}
\end{theorem}
\begin{proof}
Let the dynamical $n$-gon be regular and, for every time instant, lie in a plane parallel with the $xy$ plane. We will show that $\delta_1=\delta_2=\dots=\delta_n$ and $\gamma_1=\gamma_2=\dots=\gamma_n$, identities which, by Criterion \ref{EandU}, assure the existence of the desired homographic solution.

Without loss of generality, we assume that $\alpha_i, i=\overline{1,n},$ the constants which provide the representation \eqref{polysol} of the homographic orbits, are 
$$\alpha_1=0, \ \alpha_2={2\pi/n},\ \alpha_3={4\pi/n},\dots, \alpha_n={2(n-1)\pi/n}.$$
Then the differences $\alpha_j-\alpha_i$, $i,j=\overline{1,n},$ with $i\ne j$, can take only the values 
$$\pm 2\pi/n,\ \pm{4\pi}/{n},\dots,\pm{2(n-1)\pi/n}.$$
Consequently, these angles fully determine the constants $s_{ij}$ and $c_{ij}$ defined
in \eqref{sin-cos}.

We will first study the matrix $\Delta$. Let $s=[n/2]$ denote the integer part
of $n/2$. Then, for $n$ odd, the matrix $\Delta$ takes the circulant form  
$$
\tiny
\left[ 
\begin{array}{ccccccccccccc}
0 & a_1 & a_2 & \dots & a_{s-1} & a_s & a_s & a_{s-1} & \dots & a_3 & a_2 & a_1\\
a_1 & 0 & a_1 & \dots & a_{s-2} & a_{s-1} & a_s & a_s& \dots & a_4 & a_3 & a_2\\
a_2 & a_1 & 0 & \dots & a_{s-3} & a_{s-2} & a_{s-1} & a_s& \dots & a_5 & a_4 & a_3\\
\vdots & \vdots & \vdots & {} & \vdots & \vdots & \vdots & \vdots & {} & \vdots & \vdots & \vdots\\
a_{s-1} & a_{s-2} & a_{s-3} & \dots & 0 & a_1 & a_2 & a_3& \dots & a_{s-1} & a_s & a_s\\
a_s & a_{s-1} & a_{s-2} & \dots & a_1 & 0 & a_1 & a_2& \dots & a_{s-2} & a_{s-1} & a_s\\
a_s & a_s & a_{s-1} & \dots & a_2 & a_1 & 0 & a_1& \dots & a_{s-3} & a_{s-2} & a_{s-1}\\
a_{s-1} & a_s & a_s & \dots & a_3 & a_2 & a_1 & 0 & \dots & a_{s-4} & a_{s-3} & a_{s-2}\\
\vdots & \vdots & \vdots & {} & \vdots & \vdots & \vdots & \vdots & {} & \vdots & \vdots & \vdots\\
a_3 & a_4 & a_5 & \dots & a_{s-1} & a_{s-2} & a_{s-3} & a_{s-4} & \dots & 0 & a_1 & a_2\\
a_2 & a_3 & a_4 & \dots & a_s & a_{s-1} & a_{s-2} & a_{s-3} & \dots & a_1 & 0 & a_1\\
a_1 & a_2 & a_3 & \dots & a_s & a_s & a_{s-1} & a_{s-2} & \dots & a_2 & a_1 & 0\\
\end{array} 
\right],
$$
whereas for $n$ even, $\Delta$ becomes the circulant matrix
$$
\tiny
\left[ 
\begin{array}{cccccccccccc}
0 & a_1 & a_2 & \dots & a_{s-1} & a_s & a_{s-1} & \dots & a_3 & a_2 & a_1\\
a_1 & 0 & a_1 & \dots & a_{s-2} & a_{s-1} & a_s & \dots & a_4 & a_3 & a_2\\
a_2 & a_1 & 0 & \dots & a_{s-3} & a_{s-2} & a_{s-1} & \dots & a_5 & a_4 & a_3\\
\vdots & \vdots & \vdots & {} & \vdots & \vdots & \vdots & {} & \vdots & \vdots & \vdots\\
a_{s-1} & a_{s-2} & a_{s-3} & \dots & 0 & a_1 & a_2 & \dots & a_{s-2} & a_{s-1} & a_s\\
a_s & a_{s-1} & a_{s-2} & \dots & a_1 & 0 & a_1 & \dots & a_{s-3} & a_{s-2} & a_{s-1}\\
a_{s-1} & a_s & a_{s-1} & \dots & a_2 & a_1 & 0 & \dots & a_{s-4} & a_{s-3} & a_{s-2}\\
\vdots & \vdots & \vdots & {} & \vdots & \vdots & \vdots & {} & \vdots & \vdots & \vdots\\
a_3 & a_4 & a_5 & \dots & a_{s-2} & a_{s-3} & a_{s-4} & \dots & 0 & a_1 & a_2\\
a_2 & a_3 & a_4 & \dots & a_{s-1} & a_{s-2} & a_{s-3} & \dots & a_1 & 0 & a_1\\
a_1 & a_2 & a_3 & \dots & a_s & a_{s-1} & a_{s-2} & \dots & a_2 & a_1 & 0\\
\end{array} 
\right],
$$
where $a_1,a_2,\dots,a_s>0$ represent the $\mu_{ji}$ functions in their corresponding positions. Notice that in the $i$th row, the functions $\mu_{ji}$ with
$|i-j|=s$ and $|i-j|=n-s$, $s=\overline{1,[n/2]}$, are equal, and they are independent of $i$. But in every row of $\Delta$, there are exactly two functions $\mu_{ji}$ with $|i-j|=s$ or $|i-j|=n-s$, $s=\overline{1,[n/2]}$. For $n$ even, the single $\mu_{ji}$ term not captured in the above description corresponds to $\alpha_j-\alpha_i=\pi$ and is of the form $a_s=1/[4(1-\kappa r^2)^{3/2}]$, so it is also independent of $i$. Consequently each row of $\Delta$ contains the same elements, only differently ordered. Since all the masses are equal, we can conclude that $\delta_1=\delta_2=\dots=\delta_n$.

For $n$ odd, the matrix $\Gamma$ is circulant,
$$
\tiny
\left[ 
\begin{array}{ccccccccccccc}
0 & b_1 & b_2 & \dots & b_{s-1} & b_s & -b_s & -b_{s-1} & \dots & -b_3 & -b_2 & -b_1\\
-b_1 & 0 & b_1 & \dots & b_{s-2} & b_{s-1} & b_s & -b_s& \dots & -b_4 & -b_3 & -b_2\\
-b_2 & -b_1 & 0 & \dots & b_{s-3} & b_{s-2} & b_{s-1} & b_s& \dots & -b_5 & -b_4 & -b_3\\
\vdots & \vdots & \vdots & {} & \vdots & \vdots & \vdots & \vdots & {} & \vdots & \vdots & \vdots\\
-b_{s-1} & -b_{s-2} & -b_{s-3} & \dots & 0 & b_1 & b_2 & b_3& \dots & b_{s-1} & b_s & -b_s\\
-b_s & -b_{s-1} & -b_{s-2} & \dots & -b_1 & 0 & b_1 & b_2& \dots & b_{s-2} & b_{s-1} & b_s\\
b_s & -b_s & -b_{s-1} & \dots & -b_2 & -b_1 & 0 & b_1& \dots & b_{s-3} & b_{s-2} & b_{s-1}\\
b_{s-1} & b_s & -b_s & \dots & -b_3 & -b_2 & -b_1 & 0 & \dots & b_{s-4} & b_{s-3} & b_{s-2}\\
\vdots & \vdots & \vdots & {} & \vdots & \vdots & \vdots & \vdots & {} & \vdots & \vdots & \vdots\\
b_3 & b_4 & b_5 & \dots & -b_{s-1} & -b_{s-2} & -b_{s-3} & -b_{s-4} & \dots & 0 & b_1 & b_2\\
b_2 & b_3 & b_4 & \dots & -b_s & -b_{s-1} & -b_{s-2} & -b_{s-3} & \dots & -b_1 & 0 & b_1\\
b_1 & b_2 & b_3 & \dots & b_s & -b_s & -b_{s-1} & -b_{s-2} & \dots & -b_2 & -b_1 & 0\\
\end{array} 
\right],
$$
whereas for $n$ even, $\Gamma$ takes the circulant form
$$
\tiny
\left[
 \begin{array}{ccccccccccccc}
0 & b_1 & b_2 & \dots & b_{s-1} & 0 & -b_{s-1} & \dots & -b_3 & -b_2 & -b_1\\
-b_1 & 0 & b_1 & \dots & b_{s-2} & b_{s-1} & 0 & \dots & -b_4 & -b_3 & -b_2\\
-b_2 & -b_1 & 0 & \dots & b_{s-3} & b_{s-2} & b_{s-1} & \dots & -b_5 & -b_4 & -b_3\\
\vdots & \vdots & \vdots & {} & \vdots & \vdots & \vdots & {} & \vdots & \vdots & \vdots\\
-b_{s-1} & -b_{s-2} & -b_{s-3} & \dots & 0 & b_1 & b_2 & \dots & b_{s-2} & b_{s-1} & 0\\
0 & -b_{s-1} & -b_{s-2} & \dots & -b_1 & 0 & b_1 & \dots & b_{s-3} & b_{s-2} & b_{s-1}\\
b_{s-1} & 0 & -b_{s-1} & \dots & -b_2 & -b_1 & 0 & \dots & b_{s-4} & b_{s-3} & b_{s-2}\\
\vdots & \vdots & \vdots & {} & \vdots & \vdots & \vdots & {} & \vdots & \vdots & \vdots\\
b_3 & b_4 & b_5 & \dots & -b_{s-2} & -b_{s-3} & -b_{s-4} & \dots & 0 & b_1 & b_2\\
b_2 & b_3 & b_4 & \dots & -b_{s-1} & -b_{s-2} & -b_{s-3} & \dots & -b_1 & 0 & b_1\\
b_1 & b_2 & b_3 & \dots & 0 & -b_{s-1} & -b_{s-2} & \dots & -b_2 & -b_1 & 0\\
\end{array} 
\right],
$$
where $b_1,b_2,\dots,b_s>0$ represent the corresponding $\nu_{ji}$ functions. Thus, for $\Gamma$, we have to slightly amend the arguments used for $\Delta$: the two functions $\nu_{ji}$ with $|i-j|=s$ or $|i-j|=n-s$, which are independent of $i$ and occur in every row, are equal only in absolute value; they have opposite signs because $s_{ji}:=\sin(\alpha_j-\alpha_i)$ is odd. For $n$ even, the extra term corresponding to $\alpha_j-\alpha_i=\pi$ is $b_s=0$. Therefore $\gamma_1=\gamma_2=\dots=\gamma_n=0$. By Criterion \ref{EandU}, we can now conclude that the dynamical regular $n$-gon having equal masses at its vertices is a homographic orbit of the curved $n$-body problem. This remark completes the proof.
\end{proof}

We can now state and prove the converse of Theorem \ref{eqmasses}.

\begin{theorem}
If the masses $m_1,\dots, m_n>0$, $n\ge 3$, form a polygonal homographic solution of the curved $n$-body problem given by equations \eqref{second}, such that the polygon is regular, then $m_1=m_2=\dots=m_n$.
\end{theorem}
\begin{proof}
Notice first that, from the form of the matrix $\Gamma$ represented above (for both $n$ even and odd), if a dynamical regular $n$-gon parallel with the $xy$ plane is a homographic solution
of the curved $n$-body problem, then the masses must satisfy the system
$\Gamma{\bf m}^T={\bf 0}^T$. By Criterion \ref{secondCriterion}, the masses must also satisfy the system $\Delta{\bf m}^T={\bf d}^T$.

Since, in general, the functions $\mu_{ji}$ and $\nu_{ji}$ vary in time,
and the solutions $m_1,m_2,\dots, m_n$ of the systems $\Delta{\bf m}^T={\bf d}^T$
and $\Gamma{\bf m}^T={\bf 0}^T$ depend on $\mu_{ji}$ and $\nu_{ji}$, only
solutions for which the masses are constant and positive lead to homographic solutions.
Therefore if we fix an arbitrary time $t$ and show that the corresponding
systems with constant coefficients have solutions only when $m_1=m_2=\dots=m_n$,
then no other solutions are possible for those systems when $t$ varies.
By Theorem \ref{eqmasses}, we can then conclude that $m_1=m_2=\dots=m_n$
is the only case when the systems with variable coefficients have solutions.

So let us assume $t$ fixed and start with the system $\Delta{\bf m}^T={\bf d}^T$.
We already know from Theorem \ref{eqmasses} that this system has infinitely
many solutions, namely $m_1=m_2=\dots=m_n=\alpha$ for any $\alpha>0$,
where $\alpha$ is some function of $\delta>0$. But if we fix a value of $\delta$,
then $\alpha$ is also fixed. 

We will now show that $\det\Delta\ne 0$. To prove this fact, notice that, except for the diagonal elements, all the elements of $\Delta$ are positive. Therefore $\Delta$ is positive definite. Indeed, a simple computation shows that for any nonzero vector ${\bf z}=(z_1,z_2,\dots,z_n)\in{\mathbb R}^n$, ${\bf z}\Delta{\bf z}^T>0$. By Sylvester's criterion (see, e.g.,  \cite{Gilbert}), positive-definite matrices have positive determinants.

According to Cramer's rule, the linear system $\Delta{\bf m}^T={\bf d}^T$ has a unique
solution, which must be none else than $m_1=m_2=\dots=m_n=\alpha$, with $\alpha>0$
fixed. By Theorem \ref{eqmasses}, this solution also satisfies the system
$\Gamma{\bf m}^T={\bf 0}^T$. Since the choice of $\alpha$ depends on the choice
of $\delta$, and $\delta>0$ is arbitrarily fixed, we can draw the same conclusion
for any choice of $\delta$. Consequently a dynamical regular $n$-gon can be
a homographic solution of the curved $n$-body problem only if all masses are equal.
This remark completes the proof.
\end{proof}


\section{Non-geodesic scalene triangles}

It is natural to ask whether irregular polygons could form homographic
orbits. The conditions of Criterion \ref{EandU} for the existence and
uniqueness of dynamical $n$-gons suggest that, in general, this is not the case. Indeed, system \eqref{deltagamma} has $2n-2$ linear equations and $n$
unknowns: $m_1,m_2,\dots, m_n$. Even the case $n=3$ leads to a linear 
system of $4$ equations with $3$ unknowns, which is still unlikely to have
solutions in general.  

We will next prove that homographic orbits of the curved $3$-body problem cannot 
exist for $|z|\ne 0$ if the triangle is not equilateral. In other words, the symmetries 
of the equilateral triangle (and consequently the equality of the masses) are a 
necessary condition for the existence of such orbits. The restriction 
$|z|\ne 0$ is necessary only for  $\kappa>0$ (being automatically satisfied for $\kappa<0$) because equations \eqref{r} do not allow homographic orbits to pass 
through the equator of the sphere, since the function $1-\kappa r^2$, which appears as a denominator, cancels in that case. However, as we will see in the next section, non-equilateral relative equilibria moving along the equator do exist.

We can now state and prove the following result.

\begin{theorem}
Consider the curved $3$-body problem, given by equations \eqref{second} with
$n=3$ and masses $m_1,m_2,m_3>0$. These equations admit no homographic orbits  given by scalene non-equilateral triangles for $\kappa<0$. For $\kappa>0$, they don't admit such solutions either if the bodies stay away from the equator $z=0$.  
\end{theorem}
\begin{proof}
Notice first that any homographic orbit of the $3$-body problem must correspond to an acute triangle. Otherwise, at every time instant, there is a plane containing the $z$
axis such that all three bodies are on one side of the plane (at most two of
them in the plane). Therefore the total angular momentum vector at that time instant cannot be zero or parallel with the $z$ axis. But after rotating by $\pi$ radians, the bodies reach a position on the other side of the plane, and the angular momentum has
certainly a different direction, in violation of the angular-momentum integrals \eqref{ang}.

For the masses $m_1,m_2,m_3>0$, the conditions in Criterion \ref{secondCriterion} can be written as
\begin{equation}
\Delta{\bf m}^T={\bf d}^T\ \ {\rm and}\ \ \Gamma{\bf m}^T={\bf e}^T,
\label{3dimsys}
\end{equation}
$${\bf m}=(m_1,m_2,m_3),\ \ {\bf d}=(\delta,\delta,\delta), \ \ {\bf e}=(
\gamma,\gamma,\gamma),$$
$$\Delta=\left[ \begin{array}{ccc}
0 & a & b\\
a & 0 & c\\
b & c & 0
\end{array} \right],\ \
\Gamma=\left[ \begin{array}{ccc}
0 & u & v\\
-u & 0 & w\\
-v & -w & 0
\end{array} \right],
$$
$$a:=\mu_{21}=\mu_{12},\ b:=\mu_{31}=\mu_{13},\ c:=\mu_{32}=\mu_{23},$$
$$u:=\nu_{21}=-\nu_{12},\ v:=\nu_{31}=-\nu_{13},\ w:=\nu_{32}=-\nu_{23},$$
and $\mu_{ji}, \nu_{ji},\ i,j=1,2,3$, $i\ne j$, as defined in \eqref{munu}.

Multiplying the first equation of the system $\Gamma{\bf m}^T={\bf e}^T$ by 
$-w$, the second by $v$, and adding them, we obtain the equation
$$
-uvm_1-uwm_2=\gamma(v-w).
$$
Multiplying the third equation of that system by $-u$ and adding the
above equation to it, we get the condition $\gamma(v-u-w)=0$. Circular
permutations lead us to the system of conditions
$$\gamma(v-u-w)=\gamma(w-v-u)=\gamma(u-w-v)=0.$$
This system is satisfied either if $\gamma=0$ or when
$$
v-u-w=w-v-u=u-w-v=0.
$$
But the above equations imply that $u=v=w=0$, which is a solution with
no dynamical consequences, so necessarily $\gamma=0$. Consequently
the linear system $\Gamma{\bf m}^T={\bf e}^T$  reduces to
$$
um_2+vm_3=-um_1+wm_3=-vm_1-wm_2=0.
$$
Since $m_1,m_2,m_3>0$, the signs of $u,v$, and $w$ must be such that
\begin{equation}
(i) \  u,w>0\ {\rm and} \ v<0\ \ \ {\rm or}\ \ \ (ii)\ u,w<0\ {\rm and} \ v>0.
\label{signs}
\end{equation}

Notice that from the first equation of each of the linear system in \eqref{3dimsys}, as well as from the third equation of both systems, we can, respectively, conclude that 
$$m_2={\frac{\delta}{b}\over\frac{a}{b}-\frac{u}{v}} \ \ {\rm and} \ \
m_2={\frac{\delta}{b}\over\frac{c}{b}-\frac{w}{v}}.$$
Then
$
\frac{a-c}{b}=\frac{u-w}{v},
$
a condition that is independent of $\delta>0$.
Two similar conditions follow by circular permutations. Therefore the systems in \eqref{3dimsys} have the same set of solutions if
\begin{equation}
\frac{a-c}{b}=\frac{u-w}{v},\ \ \ \ \frac{b-a}{c}=\frac{v-u}{w},\ \ \ \ \frac{c-b}{a}=\frac{w-v}{u}.
\label{condi}
\end{equation}
We will further prove that these conditions are simultaneously satisfied only when 
$a=b=c$ and $u=v=w,$ a set of solutions  
which corresponds to equilateral triangles and, as a result, to equal
masses.

Let us first assume that the acute triangle is isosceles and not equilateral. Then, without loss of generality, we can choose $a=c\ne b$ and $u=w\ne v$. The first equation in \eqref{condi} is satisfied, whereas the other two take the
same form, namely $\frac{a-b}{a}=\frac{u-v}{u}$, which implies that $av=bu$. 
Since $a,b>0$, it follows that $u$ and $v$ have the same sign, a conclusion which contradicts \eqref{signs}.

The last case to consider is that of the scalene acute triangles that are not isosceles. Without loss of generality, we assume that $0<a<b<c$. Then 
$$
\frac{a-c}{b}<0,\ \ \frac{b-a}{c}>0,\ \ \frac{c-b}{a}>0,
$$
inequalities which, via \eqref{condi}, imply that
\begin{equation}
\frac{u-w}{v}<0,\ \ \frac{v-u}{w}>0, \ \ \frac{w-v}{u}>0.
\label{contra}
\end{equation}
But relations \eqref{signs} provide only four possible ways to order $u,v,$ and $w$, namely: (1) $v<0<u<w$, (2) $v<0<w<u$, (3) $u<w<0<v$, and (4) $w<u<0<v$.
It is easy to see that, in each case, the order and the signs of these constants 
contradict at least one of the inequalities given in \eqref{contra}, so scalene non-isosceles triangles cannot form homographic orbits either. This remark completes the proof.
\end{proof}


\section{Geodesic scalene triangles}

In this section we restrict our study to the case $\kappa>0$ and consider a situation that was not captured in the previous section, namely the relative equilibria that rotate on great circles of a sphere generated from fixed points of the equations of motion. Such orbits do not exist for $\kappa<0$ because, as shown in
\cite{Diacu1} and \cite{Diacu001}, there are no fixed points when the bodies move in hyperbolic space.  

Without loss of generality, we will analyze these orbits when the great circle is the equator, $z=0$, and the rotation takes place around the $z$ axis. This case is not captured by the system of differential equations given in equations \eqref{r} and \eqref{omega} because one denominator cancels when the bodies reach the equator. 
When $r$ is constant, i.e.\ the solution is just a relative equilibrium, instead of a
homographic orbit with both rotation and expansion and/or contraction,
the motion on the equator can be studied separately. In this case, no cancelling 
denominators show up when the motion takes place on the great circle $z=0$. Nevertheless, no two bodies can be antipodal because the corresponding configuration is a singularity of system \eqref{second} (see \cite{Diacu1} or \cite{Diacu002} for more details), and therefore the motion doesn't exist. But acute triangles moving on the
equator don't have antipodal bodies at their vertices, so singularities do not affect them.

We will first focus on fixed points lying on the equator, and show why
for every acute triangle there exist masses that provide a fixed point for the equations of motion if the bodies are placed at the vertices of the triangle. Then we explain how  relative equilibria can be generated from fixed points. By providing a large class of counterexamples, we also prove that not any three masses generate fixed points, 
and consequently not any three masses can form relative equilibria that move along the equator.  

We can now state and prove the following result.

\begin{theorem}
For any acute triangle inscribed in a great circle of the sphere $S_\kappa^2$, 
there exist bodies of masses $m_1, m_2, m_3>0$ that can be placed at the
vertices of the triangle such that they form a fixed point of system \eqref{second}
for $n=3$ and $\kappa>0$, i.e.\ for the equations of motion of the curved $3$-body problem in the case of positive curvature. 
\label{scalene}
\end{theorem}
\begin{proof}
As mentioned earlier, we can assume that the great circle is the 
equator $z=0$. To form a fixed point of system \eqref{second}, the initial conditions of
the bodies of masses $m_1,m_2$, and $m_3$ must satisfy, at the initial instant $t=0$, the constraints 
$$ 
\ddot{\bf q}_i(0)=\dot{\bf q}_i(0)=0,\ \ i=1,2,3.
$$
Let ${\bf q}_i(0)=(x_i, y_i, 0), i=1,2,3$, be the initial position of the body of mass
$m_i, i=1,2,3$, on the equator $z=0$. Using equations \eqref{second},
a straightforward computation shows that the above conditions reduce to solving the linear homogeneous algebraic system
\begin{equation}
\begin{cases}
q_{12}m_2+q_{13}m_3=0\cr
\bar{q}_{12}m_2+\bar{q}_{13}m_3=0\cr
q_{21}m_1+q_{23}m_3=0\cr
\bar{q}_{21}m_1+\bar{q}_{23}m_3=0\cr
q_{31}m_1+q_{32}m_2=0\cr
\bar{q}_{31}m_1+\bar{q}_{32}m_2=0,
\end{cases}
\label{3pairs}
\end{equation}
for $m_1, m_2$, and $m_3$, where, for $i,j=1,2,3, i\ne j$,
$$q_{ij}=\frac{x_j-a_{ij}x_i}{(1-a_{ij}^2)^{3/2}},\ \bar{q}_{ij}=\frac{y_j-a_{ij}y_i}{(1-a_{ij}^2)^{3/2}},\ a_{ij}=\kappa x_ix_j+\kappa y_iy_j.$$
But the first and the second equation of the system are linearly
dependent. Indeed, multiplying the first equation by $\kappa x_1$, the second
equation by $\kappa y_1$, and adding the two new equations, we obtain an identity.
Similarly we can prove the linear dependence of the third and fourth equation, as well as of the fifth and sixth equation. Therefore system \eqref{3pairs} can be reduced to the linear homogeneous system
\begin{equation}
\begin{cases}
q_{12}m_2+q_{13}m_3=0\cr
q_{21}m_1+q_{23}m_3=0\cr
q_{31}m_1+q_{32}m_2=0,
\end{cases}
\label{3eq}
\end{equation}
in unknowns $m_1, m_2,$ and $m_3$. 

We will further show that this system has positive solutions.
To achieve this goal we will first prove that $\det A = 0$, where $A$ is the matrix
that defines system \eqref{3eq}, namely
$$A=\left[ \begin{array}{ccc}
0 & q_{12} & q_{13}\\
q_{21} & 0 & q_{23}\\
q_{31} & q_{32} & 0\\
\end{array} \right].$$
Notice first that 
$$
\det A = q_{12}q_{23}q_{31}+q_{13}q_{21}q_{32}.
$$
But $q_{12}$ and $q_{21}$ have the same denominator, which is never zero. 
The same is true for the pair $q_{23}$ and $q_{32}$, as well as for the pair 
$q_{31}$ and $q_{13}$. Therefore to prove that $\det A=0$, it is enough to
compute the numerator $E$ of $\det A$. Notice
that $E$ has the form
$$
E=(x_2-a_{12}x_1)(x_3-a_{23}x_2)(x_1-a_{13}x_3)+(x_3-a_{13}x_1)(x_1-a_{12}x_2)
(x_2-a_{23}x_3).
$$
A straightforward computation leads to
$$
E=2(1-a_{12}a_{13}a_{23})x_1x_2x_3+(a_{12}a_{23}-a_{13})x_1^2x_2+
(a_{13}a_{23}-a_{12})x_1^2x_3
$$
$$
+(a_{12}a_{13}-a_{23})x_1x_2^2+(a_{13}a_{12}-a_{23})x_1x_3^2
$$
$$
+(a_{13}a_{23}-a_{12})x_2^2x_3+(a_{12}a_{23}-a_{13})x_2x_3^2.
$$
Without loss of generality, we can assume that $x_1=0$, i.e.\ the body
of mass $m_1$ is fixed at the point of coordinates $(x_1,y_1,z_1)=(0, \kappa^{-1/2},0)$
on the sphere $S_\kappa^2$ (see Figure \ref{fixed}). Then $E=Sx_2x_3$,
where $S=-a_{13}x_3+a_{13}a_{23}x_2+a_{12}a_{23}x_3-a_{12}x_2$. Since
$x_2, x_3\ne 0$, to prove that $\det A = 0$ it is enough to show that $S=0$. 
Using again the fact that $x_1=0$, and writing the constants $a_{ij}$ explicitly,
we obtain
$$
S=-\kappa y_1y_3x_3+\kappa^2 y_1y_3(x_2x_3+y_2y_3)x_2
+\kappa^2y_1y_2(x_2x_3+y_2y_3)x_3-\kappa y_1y_2x_2
$$
$$
=-\kappa x_3y_1y_3-\kappa x_2y_1y_2+\kappa x_3y_1y_3(\kappa x_2^2+\kappa y_2^2)+\kappa x_2y_1y_2(\kappa x_3^2+\kappa y_3^2).
$$
But since $\kappa x_2^2+\kappa y_2^2=\kappa x_3^2+\kappa y_3^2=1$,
it follows that $S=0$, therefore $\det A = 0$, so system
\eqref{3eq} has other solutions than the trivial one. 

We must still prove that among these nontrivial solutions there is a class of strictly positive solutions. For this purpose, we will analyse the signs of the constants $q_{ij}, i, j=1,2,3,$ and see how they determine the signs of $m_1, m_2,$
and $m_3$. We already assumed that $x_1=0$. Since the triangle is acute, we can 
further suppose, without loss of generality, that $x_2>0$ and $x_3<0$, which means that the body of mass $m_2$ lies in the fourth quadrant on the equatorial circle, while $m_3$ is in the third quadrant (see Figure \ref{fixed}). There are no other possibilities because the triangle is acute. (In fact we showed in \cite{Diacu1} and \cite{Diacu001} that bodies lying in the same hemisphere cannot form fixed points, thus excluding right or obtuse triangles as fixed-point candidates. Right triangles are also excluded because they form singular configurations.) 

\def\JPicScale{0.5}
\ifx\JPicScale\undefined\def\JPicScale{1}\fi
\unitlength \JPicScale mm
\begin{figure}
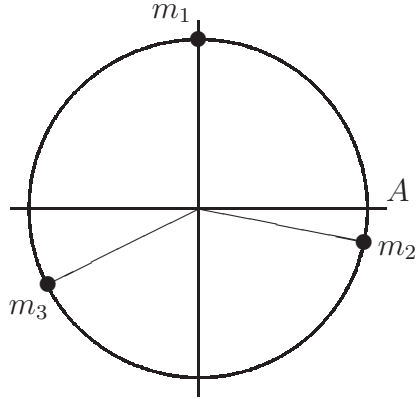


\bigskip
\caption{The fixed point solutions formed by the masses $m_1, m_2$, and $m_3$ 
on the geodesic (equator) $z=0$.}
\label{fixed}
\end{figure}

Under the above assumptions, we will show that 
\begin{equation}
q_{12}, q_{21}, q_{32} >0\ \ {\rm and}\ \ q_{13}, q_{31}, q_{23}<0.
\label{6ineq}
\end{equation} 
For this purpose, notice first that the denominators of $q_{ij}, i,j=1,2,3,$ are all positive, so we have to determine only the signs of their numerators.
Also remark that the angles $\alpha_{12}$ and $\alpha_{13}$ from the centre of the circle corresponding to the arcs $m_1m_2$ and $m_1m_3$, respectively (see Figure \ref{fixed}), both angles taken to be smaller than $\pi$, are larger than $\pi/2$. Since $a_{12}=\cos\alpha_{12}$ and $a_{13}=\cos\alpha_{13}$, it means that $a_{12}, a_{13}<0$. Then, using the fact that $x_1=0, x_2>0, x_3<0$, we can conclude that $q_{12}>0$ and $q_{13}<0$. Since $a_{12}=a_{21}$ and $a_{31}=a_{13}$, it also follows that $q_{21}>0$ and $q_{31}<0$.

To prove the last two inequalities in \eqref{6ineq}, let $\beta$ be the angle from the centre of the circle corresponding to the arc $m_2m_3$, $\alpha$ the similar angle 
corresponding to the arc $m_3A$ (see Figure \ref{fixed}), and $\gamma$ the similar angle corresponding to the arc $m_2A$, all taken to be smaller than $\pi$. Then, obviously, $\alpha=\beta+\gamma$. Also notice that $x_3=\cos\alpha, x_2=\cos\gamma,$ and $a_{23}=\cos\beta$. So 
$$x_3-a_{23}x_2=\cos\alpha-\cos\beta\cos\gamma=-\sin\beta\sin\gamma,
$$
which is negative because $0<\beta, \gamma <\pi$.
Therefore $q_{23}<0$. Finally, 
$$
x_2-a_{23}x_3=\cos\gamma-\cos\beta\cos\alpha=\sin\alpha\sin\beta,
$$
which is positive because $0<\alpha, \beta<\pi$. So $q_{23}>0$. To see now that system \eqref{3eq} has positive solutions, it is enough to notice that the two constants $q_{ij}$ showing up in each of its three equations have opposite signs. This remark completes the proof.
\end{proof}

A direct consequence of the above result is the possibility to generate 
relative equilibria from fixed points. This fact stems from the action produced by elements of the rotation group $SO(3)$ on a fixed point. From the mechanical point of view this means that we can obtain relative equilibria if we apply initial velocities, of equal speeds, tangentially to the geodesic, all oriented clockwise or all counterclockwise. This obvious remark together with Theorem \ref{scalene} prove the following result.

\begin{theorem}
Consider the curved $3$-body problem on the sphere $S_\kappa^2$, given by system \eqref{second} for $n=3$ and $\kappa>0$. Then for any acute triangle inscribed in a great circle of the sphere, there exist masses $m_1, m_2, m_3>0$ and initial velocities such that if the point particles are initially placed at the vertices of the triangle, the corresponding solution is a relative equilibrium that rotates on the great circle. 
\end{theorem}

The converse of Theorem \ref{scalene} is false, which means that not any three point particles of given masses can form a fixed point on a great circle of $S_\kappa^2$. We will provide a class of counterexamples in the case of isosceles triangles. In fact, the following result
identifies all masses for which there are no isosceles triangles that lead to fixed points
of the equations of motion of the curved $3$-body problem for $\kappa>0$.

\begin{theorem}
Consider the curved $3$-body problem on the sphere $S_\kappa^2$, given by system \eqref{second} for $n=3$, $\kappa>0$, as well as the masses $m_1=m_2=:M>0$ and $m_3=:m>0$. Assume that the initial conditions are such that the triangle
having these masses at its vertices is acute and isosceles, with the equal masses
corresponding to the base. Then, for $M\ge 4m$, no isosceles triangle
can form a fixed point.
\end{theorem}

\def\JPicScale{0.5}
\ifx\JPicScale\undefined\def\JPicScale{1}\fi
\unitlength \JPicScale mm
\begin{figure}
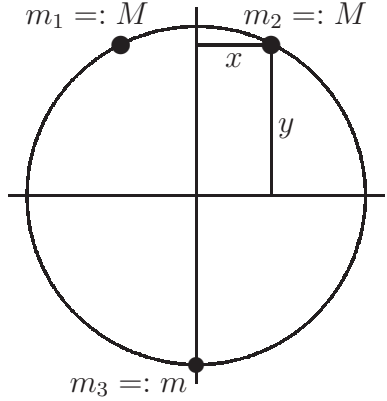


\bigskip
\caption{The initial positions of $m_1, m_2$, and $m_3$, which
form an isosceles triangle on the geodesic (equator) $z=0$.}\label{cir}
\end{figure}

\begin{proof}
Let ${\bf q}_i(0)=(x_i(0), y_i(0), 0), \ i=1,2,3,$ be the initial position of the masses with the symmetries of an isosceles triangle as in Figure \ref{cir}, i.e.\ such that 
$$x_1(0)=x,\ x_2(0)=-x,\ x_3(0)=0,$$
$$\ y_1(0)=y_2(0)=y,\ y_3(0)=-\kappa^{-1/2},$$ 
$$z_1(0)=z_2(0)=z_3(0)=0,$$
with $0<x,y<\kappa^{-1/2}$. The conditions that the coordinates form a fixed point are
$$ 
\ddot{\bf q}_i(0)=\dot{\bf q}_i(0)=0, \ \ i=1,2,3.
$$

Using these conditions in system \eqref{second}, denoting $\ddot{x}:=\ddot{x}(0),  \ddot{y}:=\ddot{y}(0)$ and asking that $\dot{x}(0)=\dot{y}(0)=0$, a straightforward computation leads us to the relations
$$
\ddot{x}=-\frac{M-4\kappa m y^2}{4\kappa^{1/2}x^2y},\ \
\ddot{y}=\frac{M-4\kappa m y^2}{4\kappa^{1/2}xy^2}.
$$

Consequently 3 bodies lying at the vertices of an isosceles triangle form
a fixed point of the equations of motion if and only if 
\begin{equation}
M=4\kappa m y^2, \ {\rm with}\ 0<y<\kappa^{-1/2}.
\label{Mm}
\end{equation}
Then some necessary conditions that an isosceles triangle forms a fixed point are
$$
0<\frac{M}{4\kappa m}<\frac{1}{\kappa}.
$$
Since $\kappa>0$, the first inequality is always satisfied. The second inequality
reduces to $M<4m$. The corresponding $x>0$ is then obtained from $\kappa x^2+
\kappa y^2=1$ and leads to the same condition. So, for $M\ge 4m$, there are no isosceles triangles that form fixed points on a great circle of ${\bf S}^2_\kappa$.
This remark completes the proof.
\end{proof}

\noindent{\it Remark}. When the triangle is equilateral, the evaluation of $\sin\varphi$, where $\varphi$ is the angle between the abscissa and the radius of the circle 
to $m_2$ (see Figure \ref{cir}), shows that $y=1/(2\kappa^{-1/2})$, so by \eqref{Mm} we can conclude that $M=m$, in agreement with what we know about Lagrangian
solutions.

\bigskip 

\noindent{\bf Acknowledgment.} The research presented in this paper was supported in part by a Discovery Grant from NSERC of Canada. 


\end{document}